\documentclass[12pt,reqno]{amsart}

\usepackage[mathlines]{lineno}

\usepackage{times}
\usepackage{amsfonts}
\usepackage{amssymb}
\usepackage{latexsym}
\usepackage{xcolor}

\newtheorem{theorem}{Theorem}
\newtheorem{lemma}[theorem]{Lemma}

\newtheorem{definition}[theorem]{Definition}
\newtheorem{remark}{Remark}
\newtheorem{corollary}[theorem]{Corollary}
\newtheorem{example}{Example}

\begin{document}

\begin{center}
\Large{On boundedness of Hausdorff-type operators on  Sobolev spaces}
\end{center}

\

\centerline{A. R. Mirotin}

\centerline{amirotin@yandex.ru}

\

\textsc{Abstract.} A new notion of a Hausdorff-type  operator on  function spaces over domains in Euclidean spaces is introduced, and a sufficient condition for the boundedness of this operator on Sobolev spaces is proved. It is shown that this condition cannot be weakened in general.  
\


\


\

\section{Introduction}

Garabedian and Rogosinskii  introduced Hausdorff  operators on the finite interval   as a natural continuous analog of the Hausdorff summation method  (see \cite[Chapter XI]{H} and references therein).
The impetus
 for the modern development of this theory was given by the work of Liflyand and  M\'{o}ricz \cite{LM} where   Hausdorff  operators on the one-dimensional real Hardy space were considered.
 
 As a result of a certain development of this scientific direction the following notion of a multidimensional Hausdorff operator over the Euclidean spaces was introduced \cite{BM}, \cite{LL}:
 
 \begin{equation}\label{H_classic}
\mathcal{H}_{\Phi,A}f (x)=\int_{\mathbb{R}^n} \Phi(u)f(A_ux)du,
\end{equation}
 where $(A_u)$ stands for a family of non-singular $n\times n$ matrices, $x\in \mathbb{R}^n$ is a column vector, and $\Phi$ is some given measurable function on $\mathbb{R}^n$ (a kernel).
The first nontrivial
results in several dimensions due to Lerner and  Liflyand \cite{LL}, see also the survey   articles  \cite{Ls}, \cite{CFW}.

Later the generalizations of the aforementioned definition were given for locally compact  groups,   homogeneous and double coset spaces of such groups instead of the Euclidean space in \cite{MCAOT, JMAA, Nachr, homogen,  Lie, JOTH}. For the case of a disc 
 in $\mathbb{C}$  see, e.~g.,  \cite{MCAOT}, \cite{KM}, \cite{KGM}. The main observation which leads to these generalizations is that the mappings $x\mapsto Ax$ where $A\in {\rm GL}(n,\mathbb{R})$ form the whole group of automorphisms of the additive group $\mathbb{R}^n$. So, for any set $\Omega$ which is an object of some category one can define Hausdorff operators over $\Omega$ using   automorphisms of this category (see \cite{arx} for details). 

In this paper, guided by this idea we consider a domain   $\Omega\subseteq \mathbb{R}^n$ endowed with Euclidean metric and a family $(A(u))_{u\in T}\subseteq {\rm Iso}(\Omega)$ of surjective isometries of $\Omega$. Thus, we arrive to the following definition.

\begin{definition} 
Let $(T,\mu)$ be a measure space and  $(A(u))_{u\in T}\subseteq {\rm Iso}(\Omega)$ be a $\mu$-measurable family of surjective isometries of a domain  $\Omega\subseteq \mathbb{R}^n$  (with respect to the  compact-open topology in ${\rm Iso}(\Omega)$). Let $\Phi:T\to \mathbb{C}$ be a $\mu$-measurable function. Then for a $\mu$-measurable function $f:\Omega\to \mathbb{C}$ we put
\begin{equation}\label{H_mu}
\mathcal{H}_{\mu, \Phi,A}f (x):=\int_{T} \Phi(u)f(A(u)(x))d\mu(u)
\end{equation}
(the Lebesgue integral).
\end{definition}

The main result of the paper states that for ``good'' domains the condition $\Phi\in L^1(\mu)$ is sufficient for the boundedness of an operator $\mathcal{H}_{\mu, \Phi,A}$  in Sobolev spaces  $W^{1,p}(\Omega)$.  We show also that this condition (which is  obviously necessary for bounded $\Omega$)   cannot be weakened in general for unbounded domains as well.  So, this work may be considered as a contribution to the further development of the theory of Hausdorff operators.

\begin{remark}
{\rm
In the pioneering  work  \cite{ZhaoGuo}  by G. Zhao and W. Guo the boundedness of Hausdorff operators of the form \eqref{H_classic} on the Sobolev spaces  $W^{k,1}( \mathbb{R}^n)$ ($k\in \mathbb{N}$) was studied for the first time.  Our approach is  different  from the approach in \cite{ZhaoGuo} because

1) we consider spaces $W^{1,p}(\Omega)$ instead of $W^{k,1}(\mathbb{R}^n)$;

2) we consider general subdomains $\Omega$  in $\mathbb{R}^n$;

3) our definition of  
a Hausdorff operator differs from the  definition used in \cite{ZhaoGuo} (this is true even in the case of the space  $W^{1,1}(\mathbb{R}^n)$);

 4) our methods of proof  are  different from ones used in \cite{ZhaoGuo}.
}
\end{remark}

\section{ Main result}
\label{sec:sobolev}\setcounter{equation}{0}

We need some preparations to prove our main result.

The following lemma has shown itself to be a universal means of proving the boundedness of Hausdorff-type operators in various functional spaces, e.g., \cite{MCAOT, JMAA, Nachr, homogen,  Lie, JOTH}.

\begin{lemma}\label{lm1} \cite[Lemma 2]{JMAA} \textit{Let $(\Omega;\nu)$ be a measure space,
$\mathcal{F}(\Omega)$
 some Banach space of $\nu$-measurable functions on $\Omega$,  $(T,\mu)$ a $\sigma$-compact quasi-metric space with  positive Radon measure $\mu$, and $F(u, x)$ a function on $T\times \Omega$. Assume that}

(a) \textit{the convergence of a sequence in norm in $\mathcal{F}(\Omega)$ yields the convergence of some
subsequence to the same function for $\nu$-a.~e. $x\in \Omega$; }

(b)  \textit{$F(u, \cdot) \in \mathcal{F}(\Omega)$ for $\mu$-a.~e. $u\in T$;}

 (c) \textit{the map  $u\mapsto F(u, \cdot):T\to \mathcal{F}(\Omega)$ is Bochner integrable with respect to } $\mu$.

  \textit{Then for $\nu$-a.~e. $x\in  \Omega$ one has}
$$
\left((B)\int_T F(u, \cdot)d\mu(u)\right)(x)=
\int_T F(u, x)d\mu(u)
$$
($(B)$ stands for the Bochner integral in $\mathcal{F}(\Omega)$).
\end{lemma}

We shall use also the next lemma for the proof of our main result.

\begin{lemma}\label{lem2}
 Each isometry $A$ of a domain $\Omega\subseteq \mathbb{R}^n$ preserves the Lebesgue measure.

\end{lemma}

\begin{proof} It follows from the main result in  \cite{Zeleny} (see  p. 433  therein) that
 each  compact subset $K$ of a domain $\Omega$ can be generated
from closed Euclidean  balls that contain in  $\Omega$  by countable monotone unions, countable monotone intersections,
and countable disjoint unions. Since an  isometry $A$ preserves  Lebesgue measures of balls in $\Omega$, it preserves  Lebesgue measures of countable monotone unions, countable monotone intersections,
and countable disjoint unions of such balls (due to the  continuity and to the sigma-additivity of a measure). As a consequence, it preserves the Lebesgue measure of any  compact subset $K$ of $\Omega$.
\end{proof}

\begin{corollary}
Let $\Phi\in L^1(\mu)$ and $1\le p<\infty$. Then the corresponding Hausdorff-type operator $\mathcal{H}_{\mu, \Phi,A}$  is bounded in the Lebesgue space $L^{p}=L^{p}(\Omega)$ and
$\|\mathcal{H}_{\mu, \Phi,A}\|_{L^p\to L^p}\le \|\Phi\|_{L^1(\mu)}$.
\end{corollary}

\begin{proof} Since a composition with an isometry preserves the $L^p$-norm, we have by the Minkovski inequality
$$
\|\mathcal{H}_{\mu, \Phi,A}f\|_{L^p}\le \int_T|\Phi(u)|\|f\circ A(u)\||_{L^p}d\mu(u)= \|\Phi\|_{L^1(\mu)}\|f\|_{L^p}.
$$
\end{proof}

In the following $\Omega$ stands for a  domain in $\mathbb{R}^n$  endowed with the Euclidean metric $\rho$  and Lebesgue measure $d\nu(x)=dx$, and $\rm{Iso}(\Omega)$  denotes the group of bijective isometries of $\Omega$. Every such isometry is  smooth by the Myers--Steenrod theorem \cite{MyersSteenrod}.  Equipped  with the compact-open topology $\rm{Iso}(\Omega)$ becomes a locally compact topological group \cite[p. 46, Theorem 4.7]{KobNom}.  We call a family $(A(u))_{u\in T}\subseteq \rm{Iso}(\Omega)$ measurable if the map $u\mapsto A(u)$ is  measurable as a map between  the measure space $(T,\mu)$ and  the topological space $\rm{Iso}(\Omega)$. We shall wright
$$
A(u)(x)=(a_1(u)(x),\dots,a_n(u)(x))\quad (u\in T, x\in \Omega)
$$
where $a_k(u)\in C^1(\Omega)$  for all $k\in \{1,\dots,n\}, u\in T$ .

Recall  that the Sobolev space $W^{1,p}(\Omega)$ ($1 \le p <\infty$) is the Banach space of functions
$f\in L^p(\Omega)$ with  weak derivatives $\partial_k f:=\frac{\partial f}{\partial y_k}\in L^p(\Omega, \mathbb{R}^n)$ ($k=1,\ldots,n$) equipped with the norm
$$
\|f\||_{W^{1,p}}:=\|f\||_{L^p}+\sum_{k=1}^n\left\|\partial_k f\right\|_{L^p}
$$
(see, e.~g., \cite{Adams} for details). 

According to \cite{Heinonen} a domain $\Omega\subseteq \mathbb{R}^n$ is called {\it a $W^{1,p}$-extension domain} if for every Sobolev 
function $f\in W^{1,p}(\Omega)$ there is a Sobolev function $\tilde f\in W^{1,p}(\mathbb{R}^n)$ such that $ f=\tilde f$
almost everywhere in $\Omega$  and that 
$$
\|\tilde f\|_{W^{1,p}}\le C(\Omega)\|f\|_{W^{1,p}}
$$
where the norm on the left-hand side is the Sobolev norm in $W^{1,p}(\mathbb{R}^n)$ and that on 
the right-hand side is the Sobolev norm in $W^{1,p}(\Omega)$. The whole space $\mathbb{R}^n$ and every its bounded subdomain with 
smooth boundary is an extension domain, but an extension domain can be quite 
complicated. See \cite[Example 8.24 (f)]{Heinonen}.

\begin{theorem}\label{Sobolev} 
Let $T$
 be a $\sigma$-compact quasi-metric space with  Radon
measure $\mu$. Let $\Omega$ be a  $W^{1,p}$-extension domain in $\mathbb{R}^n$,   and a  family  $(A(u))_{u\in T}\subseteq \rm{Iso}(\Omega)$ be measurable.

(i) The Hausdorff-type operator $\mathcal{H}_{\mu, \Phi,A}$  is bounded in the Sobolev space $W^{1,p}(\Omega)$ if  $\Phi\in L^1(\mu)$ ($1<p<\infty$), and  in this case
$\|\mathcal{H}_{\mu,\Phi,A}\|\le c\|\Phi\|_{L^1(\mu)}$ for some positive constant $c=c(\Omega,p)$.

(ii) Let in addition  for some constant $C>0$ and for all $u\in T$ one has $\|\partial a_k(u)(\cdot)/\partial x_j\|_{L^\infty(\Omega)}\le C$ ($k, j\in \{1,\dots,n\}$). Then $\mathcal{H}_{\mu, \Phi,A}$  is bounded in  $W^{1,1}(\Omega)$, $\|\mathcal{H}_{\mu,\Phi,A}\|\le (Cn+1)\|\Phi\|_{L^1(\mu)}$,  and for a function $f(y)$ in  $W^{1,p}(\Omega)$ ($1\le p<\infty$) we have  for all $j\in \{1,\dots,n\}$ 
\begin{align}\label{partial}
\frac{\partial}{\partial x_j}\mathcal{H}_{\mu,\Phi,A}f(x)&=\int_T\Phi(u)\frac{\partial}{\partial x_j}f(A(u)(x))d\mu(u)\\ \nonumber
&=\int_T\Phi(u)\sum_{k=1}^n\frac{\partial f}{\partial y_k}(A(u)(x))\frac{\partial}{\partial x_j}a_k(u)(x)d\mu(u).
\end{align}
Consequently,
\begin{align*}
\nabla \mathcal{H}_{\mu,\Phi,A}f(x)&=\int_T\Phi(u)\nabla  f(A(u)(x))d\mu(u).
\end{align*}

\end{theorem}

\begin{proof} (i) Let $\Phi\in L^1(\mu)$.
It is known (see, e.g., \cite[p. 40]{Heinonen})
that $W^{1,p}(\Omega)$ equals the Haj{\l}asz-Sobolev space $M^{1,p}(\Omega)$ in a sense  that 
a function $f\in L^p(\Omega)$ belongs to  $W^{1,p}(\Omega)$  if and only if there exists 
$g\in L^p(\Omega)$ so that
\begin{align}\label{Hajlasz}
|f(x)-f(y)|\le \rho(x,y)(g(x)+g(y))
\end{align}
holds for a.~e. $x, y\in \Omega$ (recall that $\rho$ stands for the Euclidean distance in $\mathbb{R}^n$). 
In this case the following norm
$$
\|f\|_M:=\|f\|_{L^p}+\inf \|g\|_{L^p}
$$
where infimum is taken over all $g\in L^p(\Omega)$ satisfying the defining inequality \eqref{Hajlasz} is comparable with the usual norm in $W^{1,p}(\Omega)$.

 We shall verify the conditions of Lemma \ref{lm1} where  $\mathcal{F}(\Omega)=W^{1,p}(\Omega)$, $F(u,x)=\Phi(u)f(A(u)(x))$. 
 
 (a) Since $\|f\|_{L^p}\le \|f\|_M$, this follows from the well known F. Riesz theorem.
 
 (b) Note that every $A\in \rm{Iso}(\Omega)$ preserve the Lebesgue measure on $\Omega$ by Lemma \ref{lem2}.  For  $f\in W^{1,p}(\Omega)$ and $A\in \rm{Iso}(\Omega)$ we have for a.~e. $x, y\in \Omega$
\begin{align*}
|f(A(x))-f(A(y))|&\le \rho(A(x),A(y))(g(A(x))+g(A(y)))\\ 
&=\rho(x,y)(g_1(x)+g_1(y)),  \nonumber
\end{align*} 
where $g_1:=g\circ A\in  L^p(\Omega)$. Thus, $f\circ A(u)\in  W^{1,p}(\Omega)$ for all $u\in T$ and (b) holds.

(c)  We shall use the criterium of Bochner integfrability (see, e.~g., \cite{Hyt}). The space $W^{1,p}(\Omega)$ is separable  (see, e.~g., \cite[Theorem 3.6]{Adams}). Therefore to verify  that the $W^{1,p}(\Omega)$-valued function $u\mapsto F(u,\cdot)$ is strongly $\mu$-measurable it suffices to prove that the $W^{1,p}(\Omega)$-valued function $u\mapsto f\circ A(u)$  is weakly $\mu$-measurable.
To do this, note that every linear bounded functional $L$ on $W^{1,p}(\Omega)$ has the form
$$
L(f)=\sum_{j=1}^n\int_{\Omega}\frac{\partial}{\partial x_j}f(x)v_j(x)dx
$$
for some  functions $v_j\in L^{p'}(\Omega)$ ($1/p+1/p'=1$) see, e.~g., \cite[Theorem 3.19]{Adams}. Thus, as  a family  $(A(u))_{u\in T}$  is measurable and all  operators $\partial/\partial x_j$ are  continuous on $W^{1,p}(\Omega)$, the function 
$$
u\mapsto L( f\circ A(u))=\sum_{j=1}^n\int_{\Omega}\frac{\partial}{\partial x_j}f(A(u)(x))v_j(x)dx
$$
is measurable, too. Indeed,  for each $x\in \Omega$ the evaluation map $A\mapsto A(x)$, $\rm{Iso}(\Omega)\to \Omega$ is continuous with respect to  the compact-open topology in $\rm{Iso}(\Omega)$, and thus the map $u\mapsto f\circ A(u)(x)$  is measurable as a composition of measurable mappings. 
 
Next, for each  $A\in \rm{Iso}(\Omega)$ and every $f\in W^{1,p}(\Omega)$ we have $f\circ A\in L^{p}(\Omega)$, $\|f\|_{L^p}=\|f\circ A\|_{L^p}$, and
$$
\|f\circ A\|_M:=\|f\circ A\|_{L^p}+\inf \|g_1\|_{L^p}=\|f\|_{L^p}+\inf \|g_1\|_{L^p}
$$
where infimum is taken over all $g_1\in L^p(\Omega)$ satisfying the condition
$$
|f(A(x))-f(A(y))|\le\rho(x,y)(g_1(x)+g_1(y))
$$
for a.~e. $x,y\in \Omega$. Since every such function has the form $g_1=g\circ A$ where $g=g_1\circ A^{-1}$ satisfies the condition \eqref{Hajlasz}, and $\|g_1\|_{L^p}=\|g\|_{L^p}$, we have $\|f\circ A\|_M=\|f\|_M$. This implies (c) due to the criterium of Bochner integfrability,  because $\Phi\in L^1(\Omega)$ and  $\|\Phi(u)f\circ A(u)\|_M=|\Phi(u)|\|f\|_M$ for all $u\in T$.

Thus, by  Lemma \ref{lm1}, 
\begin{equation*}
\mathcal{H}_{\mu, \Phi,A}f =\int_{T} \Phi(u)f\circ A(u)d\mu(u)
\end{equation*}
(the Bochner integral for $W^{1,p}(\Omega)$). 
Therefore
\begin{eqnarray*}
\|\mathcal{H}_{\mu, \Phi,A}f\|_{M}& \le&\int_{T} |\Phi(u)|\|\|f\circ A\|_M\|d\mu(u)\\
&=&\|\Phi\|_{L^1(\mu)}\|f\|_M.
\end{eqnarray*}
Since the norm $\|\cdot\|_M$ is comparable with the usual norm in $W^{1,p}(\Omega)$, the assertion (i) follows.

(ii) Let $1\le p<\infty$.  We shall use  the following special case of the  definition of the  generalized partial derivative of a    locally  integrable function on  $\Omega$. Let $h, g\in L^1_{\rm loc}(\Omega)$.  If
$$
\int_\Omega h(x)\frac{\partial \varphi(x)}{\partial x_j}dx=-\int_\Omega  g(x)\varphi(x)dx \mbox{ for all } \varphi\in \mathcal{D}(\Omega)
$$
where $\mathcal{D}(\Omega)$ is the space of test functions  (indefinitely differential functions with compact support  in $\Omega$) then we say that $\partial h/\partial x_j = g$ (in the 
sense of distributions).

 Now let $\Phi\in L^1(\mu)$ and $ \varphi\in \mathcal{D}(\Omega)$. Then by the Fubuni theorem
 \begin{align}\label{deriv1}
 \int_\Omega (\mathcal{H}_{\mu,\Phi,A}f)(x)\frac{\partial \varphi(x)}{\partial x_j}dx&= \int_\Omega \int_{T}\Phi(u)f(A(u)(x)\frac{\partial \varphi(x)}{\partial x_j}d\mu(u)dx\\ \nonumber
 &= \int_{T} \int_\Omega \Phi(u)f(A(u)(x)\frac{\partial \varphi(x)}{\partial x_j}dxd\mu(u).
 \end{align}
The application of  the Fubuni theorem is justified by the following estimate
$$
 \int_{T} \int_\Omega |\Phi(u)||f(A(u)(x)|\left|\frac{\partial \varphi(x)}{\partial x_j}\right|dxd\mu(u)
 $$
 $$
 \le \left\|\frac{\partial \varphi(x)}{\partial x_j}\right\|_{L^\infty(\Omega)} \int_{T} |\Phi(u)|\int_\Omega |f(A(u)(x)dxd\mu(u)
 $$
 $$
 =\left\|\frac{\partial \varphi(x)}{\partial x_j}\right\|_{L^\infty(\Omega)} \|f\|_{L^1(\Omega)}\||\Phi\|_{L^1(\mu)}<\infty.
 $$
Since

$$
 \int_{T}  |\Phi(u)| \int_\Omega \sum_{k=1}^n\left|\frac{\partial f(A(u)(x))}{\partial y_k}\right| \left|\frac{\partial a_k(u)(x))}{\partial x_j}\right||\varphi(x)|dxd\mu(u)
$$
\begin{align}\label{estim}
&\le C  \|\varphi \|_{L^\infty(\Omega)} |\Phi\|_{L^1(\mu)}\sum_{k=1}^n \int_\Omega\left|\frac{\partial f(A(u)(x))}{\partial y_k}\right|dx\\ \nonumber
&\le C  \|\varphi \|_{L^\infty(\Omega)} \|\Phi\|_{L^1(\mu)}\|f\|_{W^{1,p}}<\infty,
\end{align}
we deduce from \eqref{deriv1} (again  by the Fubuni theorem) that for all $ \varphi\in \mathcal{D}(\Omega)$
$$
 \int_\Omega (\mathcal{H}_{\mu,\Phi,A}f)(x)\frac{\partial \varphi(x)}{\partial x_j}dx=
$$
$$
 \int_\Omega\left(-\int_T \Phi(u)\sum_{k=1}^n\frac{\partial f(A(u)(x))}{\partial y_k}\frac{\partial a_k(u)(x))}{\partial x_j}d\mu(u)\right)\varphi(x)dx,
$$
which proves \eqref{partial}.

To finish the proof first note that
\begin{align*}
\|\mathcal{H}_{\mu,\Phi,A}f\||_{L^{1}}&\le \int_T |\Phi(u)|\|f\circ A(u)\|_{L^{1}}d\mu(u)\\
&=\|\Phi\|_{L^1(\mu)}\|f\|_{L^{1}}\le \|\Phi\|_{L^1(\mu)}\|f\|_{W^{1,1}}.
\end{align*}
Similarly, formula \eqref{partial}  implies that for all $j$
\begin{align*}
\|\partial_j \mathcal{H}_{\mu,\Phi,A}f\||_{L^{1}}&\le C\int_T |\Phi(u)|\sum_{k=1}^n\|(\partial_k f)\circ A(u)\|_{L^{1}}d\mu(u)\\
&= C\|\Phi\|_{L^1(\mu)}\sum_{k=1}^n\|\partial_k f\|_{L^{1}}\le C \|\Phi\|_{L^1(\mu)}\|f\|_{W^{1,1}}.
\end{align*}

Therefore
\begin{align*}
\|\mathcal{H}_{\mu,\Phi,A}f\||_{W^{1,1}}&=\|\mathcal{H}_{\mu,\Phi,A}f\||_{L^1}+\sum_{j=1}^n\left\|\partial_j \mathcal{H}_{\mu,\Phi,A}f\right\|_{L^1}\\
&\le(nC+1)\|\Phi\|_{L^1(\mu)}\|f\|_{W^{1,1}}
\end{align*}
which completes the proof.

\end{proof}

\begin{remark}
{\rm If the domain $\Omega$ is bounded then $1\in W^{1,p}(\Omega)$. So, in this case the condition  $\Phi\in L^1(\mu)$ is necessary for the boundedness of $\mathcal{H}_{\mu,\Phi,A}$ in  
$W^{1,p}(\Omega)$; for the unbounded case see Corollary \ref{Rn}.  
}
\end{remark}

\begin{remark}
{\rm
One can replace the condition on $\Omega$ to be a  $W^{1,p}$-extension domain  in the Theorem \ref{Sobolev}  by the condition that $\Omega$  supports a $p$-Poincar\'e inequality in a sense of \cite{HK}. Indeed, in this case the equality $W^{1,p}(\Omega)=M^{1,p}(\Omega)$ remains true by \cite[Corollary 10.2.9 and Theorem 7.4.5]{HK}.
}
\end{remark}

\begin{example} {\rm Let $T={\rm O}(n)$ be the group of real orthogonal $n\times n$  matrices, $A(u)(x)=ux$ ($u\in {\rm O}(n)$,  $x$ is a column vector), $\Phi\equiv 1$, and $\mu$ the normalized Haar measure of the (compact) group ${\rm O}(n)$. In this case, the Hausdorff operator turns into  the following integral transform (the averaging  operator over the action of ${\rm O}(n)$)
$$
\mathcal{A}f(x)=\int_{{\rm O}(n)}f(ux)d\mu(u).
$$
This operator is  bounded in $W^{1,p}(B)$ ($1<p<\infty$) where $B$  is the ball in $\mathbb{R}^n$ (possibly 
with infinite radius so that $B=\mathbb{R}^n$)  by Theorem \ref{Sobolev}(i)    and
$$
\frac{\partial}{\partial x_j}\mathcal{A}f(x)=\int_{{\rm O}(n)}\sum_{k=1}^n\frac{\partial f(ux)}{\partial y_k}u_{kj}d\mu(u),
$$
since in this case  $a_k(u)(x)=\sum_{j=1}^nu_{kj}x_j$ where $u=(u_{kj})$. Moreover, this operator is bounded in $W^{1,1}(B)$ and  $\|\mathcal{A}\|\le n+1$ by Theorem \ref{Sobolev}(ii), because $|u_{kj}|\le 1$ for all $k,j$ for $u=(u_{kj}) \in {\rm O}(n)$.

A similar result is valid for an  averaging  operator
$$
\mathcal{A}_T f(x)=\frac{1}{\mu(T)}\int_{T}f(ux)d\mu(u)
$$
where $T\subset {\rm O}(n)$ stands for a measurable set with $\mu(T)>0$.
}
\end{example}

It is well known that each  Euclidean motion in $\mathbb{R}^n$ has the form $A(x)=Vx+b$, where $V$ is an orthogonal  matrix, $V\in {\rm O}(n)$, and $b\in \mathbb{R}^n$ ($x$ is a column vector). Thus,  the following corollary of Theorem \ref{Sobolev}   is true.

\begin{corollary}\label{Rn} Let $V_u\in {\rm O}(n)$, $b_u\in \mathbb{R}^n$ for all $u\in T$, and a  family  $A(u)(x):=V_ux+b_u$ of isometries of $\mathbb{R}^n$ be  measurable. 

(i) If $\Phi\in L^1(\mu)$ then the operator
\begin{eqnarray}\label{HRn}
 \mathcal{H}_{\mu,\Phi,A}f(x)=\int_{T}\Phi(u)f(V_u x+b_u)d\mu(u)
 \end{eqnarray}
  is bounded in $W^{1,p}(\mathbb{R}^n)$ ($1<p<\infty$), and $\|\mathcal{H}_{\mu,\Phi,A}\|\le c\|\Phi\|_{L^1(\mu)}$ for some positive constant $c=c(n,p)$. Moreover, this operator is bounded in $W^{1,1}(\mathbb{R}^n)$, and in this case $\| \mathcal{H}_{\mu,\Phi,A}\|\le (n+1)\|\Phi\|_{L^1(\mu)}$.
  
(ii)  If the range of the map $u\mapsto b_u$ is bounded in $ \mathbb{R}^n$  the condition $\Phi\in L^1(\mu)$  is necessary for the boundedness of $\mathcal{H}_{\mu,\Phi,A}$ in  
$W^{1,p}(\mathbb{R}^n)$.
\end{corollary}

\begin{proof} (i) This follows from the previous theorem (since  $V_u\in {\rm O}(n)$, one can choose $C=1$ in this theorem as in  Example 1).

(ii)  Let $\|b_u\|\le C_1$  for all $u\in T$ and  let  $f(y)=e^{-\|y\|^2}$.  Then $f\in W^{1,p}(\mathbb{R}^n)$ and 
\begin{align*}
f(V_ux+b_u)&=e^{-\|V_ux+b_u\|^2}\ge e^{-2(\|V_ux\|^2+\|b_u\|^2)}\\
&= e^{-2(\|x\|^2+\|b_u\|^2)}\ge e^{-2(\|x\|^2+C_1^2)}
\end{align*} 
for all $u\in T$. Therefore the convergence of a Lebesgue integral  in \eqref{HRn} implies
$$
\infty>\int_{T}|\Phi(u)|f(V_u x+b_u)d\mu(u)\ge e^{-2(\|x\|^2+C_1^2)}\int_{T}|\Phi(u)|d\mu(u),
$$
and thus $\Phi\in L^1(\mu)$.
\end{proof}

\

\section{acknowledgments}
The author is partially supported by the State Program of Scientific Research
of Republic of Belarus, project No. 20211776,
 and by the Ministry of Education and Science of Russia,  agreement No. 075-02-2023-924.

\

\section{Data availability statement}
This article has no additional data.
This work does not have any conflicts of interest.

\section{Disclosure Statement}
The author confirms that there are no relevant financial or non-financial competing interests to report.

Department of Mathematics and Programming Technologies, Francisk Skorina Gomel State University, Gomel, 246019, Belarus $\&$ Regional Mathematical Center, Southern Federal
University, Rostov-on-Don, 344090, Russia.

\

Declaration of AI use. I have not used AI-assisted technologies in creating this article.

\end{document}